\documentclass[12pt,reqno]{amsart}

\usepackage[mathlines]{lineno}

\usepackage{times}
\usepackage{amsfonts}
\usepackage{amssymb}
\usepackage{latexsym}
\usepackage{xcolor}

\newtheorem{theorem}{Theorem}

\newtheorem{definition}[theorem]{Definition}
\newtheorem{remark}{Remark}
\newtheorem{corollary}[theorem]{Corollary}
\newtheorem{example}{Example}

\usepackage{comment}

\begin{document}

\begin{center}
\Large{On the regularity  of generic Hausdorff-type transformations}
\end{center}

\

\centerline{A. R. Mirotin}

\

\centerline{amirotin@yandex.ru}

\

\

Abstract. The general notion of a Hausdorff-type operator with a kernel depending on an external variable is introduced and  generalizations and analogs of  classical results  on the regularity of various  summation methods are proved for the case of such operators.

\
\begin{comment}
Key wards. Hausdorff operator, Hausdorff-type operator, average of a function, summation method, regular transformation, locally compact group.

\
 2020 MSC: Primary 40C10; Secondary 43A55,  47G10,  45P05, 47B38. 
\end{comment}
\ 

\section{Introduction and preliminaries}

As is well known,  results on the regularity of transformations of the form
\begin{equation}\label{tm}
t(m)=\sum_{n=0}^\infty c_{m,n}f_n\ \ (m\in \mathbb{Z}_+)
\end{equation}
form an important part of the general theory of summation methods (e.g. \cite[Chapter III]{H}\footnote{G.H. Hardy: ``The most important  
transformations are regular''  \cite[p. 43]{H}.}). The classical theorem, due to Toeplitz,  Schur, and Silverman states necessary and sufficient conditions for the regularity of such transformations (e.g., \cite[Chapter III, \S 3.3, Theorem 2]{H},  \cite{Wilansky}). For generalizations and analogs of this result see e.g.,  \cite[Chapter III, \S 3.5, Theorems 5, 6]{H}.

On the other hand, Hausdorff  operator on the semi-axis  was introduced by  Rogosinski and independently by Garabedian in the form
\begin{equation}\label{hardy}
(H_\mu f)(x)=\int_0^1f(ux)d\mu(u),\ x\in (0,\infty),
\end{equation}
where $\mu$ stands for a finite measure on $[0,1]$ \cite{Rogosinski}, \cite{FuchsRogosinski}.
This ``continuous  Hausdorff  method of summation`` is  a natural  analog of the classical Hausdorff summation method  (e.g., \cite[Chapter XI, p.  276]{H}).  The  Abel, the Ces\'aro,  and  the H$\ddot {\rm o}$lder means of a function $f$ on $(0,\infty)$ of all real and positive orders have this form. Typical means are obtained also by restricting to the choice of $f$ in \eqref{hardy} to step functions. This idea was used in  \cite{FuchsRogosinski} in order to apply continuous  Hausdorff  method of summation to the study of summation by these means.

Rogosinski   proved the following result on the regularity of the transformation  $H_\mu$  (e.g., \cite[Chapter XI, Theorem 217]{H}).

\begin{theorem}\label{th:1}  In order that the transformation \eqref{hardy} should be 
regular, i.e. that $f(x) \to l$ should imply $(H_\mu f)(x)\to l$ ($x\to\infty$), it is necessary and 
sufficient that $\mu([0,1])=1$ and $\mu(\{0\})=0$. 
\end{theorem}

This is an  analog of the regularity result for  classical Hausdorff summability.

The aim of this note is to put the results mentioned above to a   general context.
 We will work  in the context of the  generalized Hausdorff-type operators   with kernels depending on an external variable. In general, these operators act between two different sets. To the best of the author's knowledge, Hausdorff operators in such a generality have not been considered before. It turns out that these very advanced Hausdorff-type operators continue to preserve regularity under some natural assumptions.

In resent two  decades a lot of different notions of a Hausdorff operator have been suggested (see \cite{BM2, CFW,    KMHB, KSZhu, LL, Ls, JMAA, MCAOT, MCAOT2, Lie,  JOTH, Nachr, SG, S}, the survey article \cite{ LM:survey}, and the bibliography therein).
The following definition  covers all known classes of operators bearing the name ``Hausdorff'' and many classical and new types of operators and transformations such as  transformations arise in  classical summation methods,  classical and  discrete Hilbert transforms and there generalizations, integral Hankel operators, orbital integrals,  convolution operators on groups, Hadamard-Bergmann convolutions etc. (e.g., \cite{ LM:survey},  \cite{KAM}). As was mentioned above, the characteristic features of this definition are the consideration of kernels depending on an external variable and the action between two different sets.

\begin{definition}\label{general1}  Let  $S$ and  $S'$ be two sets, $(\Omega,\mu)$ denotes some measure space, and $A(u): S\to S'$ ($u\in  \Omega $) be some family of mappings\footnote{$A(u)$ do not assumed to be invertible.}.  Let  $\Phi(u,x)$ be a given  function on $\Omega\times S$ which is $\mu$-measurable for every $x\in S$, and $V$ some Banach space.
 {\it A Hausdorff-type operator} acts on a  functions $f: S' \to V$  by the rule
\begin{equation}\label{general}
(\mathcal{H}_{\Phi,A,\mu}f)(x) =\int_{\Omega} \Phi(u,x)f(A(u)(x))d\mu(u), \ x\in S,
\end{equation}
%\end{comment}
provided the integral  converges in a suitable  sense.
\end{definition}

\begin{remark}\label{rem:1}
By the well known criterion of Bochner integrability the integral  in \eqref{general} exists in the sense of Bochner (with respect to the Banach space $V$) for every $x\in S$ if for  fixed $x$ the function $\Phi(\cdot, x)\in L^1(\mu)$, $V$-valued function $f$ is bounded, and  the map $u\mapsto f(A(u)(x))$, $\Omega\to V$ is $\mu$-measurable.
\end{remark}

\begin{remark}\label{rem:10}
If a kernel $\Phi(u,x)=\Phi(u)$ does not depend of the external variable $x$ we call the corresponding Hausdorff operator ``Hausdorff operator  with a one-variable kernel''.  Such operators were at the first time introduced in  \cite{KL} under the name  ``a broad Hausdorff operator''.
\end{remark}

Obviously,  the operator \eqref{hardy} is a very special case of \eqref{general}. 
%On the connections between  \eqref{tm} and \eqref{general} see Remark \ref{on(i)}.

An interesting special case of a Hausdorff-type operator in a sense of Definition \ref{general1} appears if we take $\Omega =\mathbb{Z}_+$. Denoting $c_n(x):=\Phi(n,x)$, $A_n(x):=A(n,x)$, and $\mu_n:=\mu(\{n\})$ ($n\in \mathbb{Z}_+$) we obtain a {\it discrete  Hausdorff-type operator} in a form

\begin{equation}\label{discrete}
(\mathcal{H}_{c,A,\mu}f)(x) =\sum_{n=0}^\infty c_n(x)\mu_nf(A_n(x)),
\end{equation}
provided the series  converges for   $x\in S$.

If  $A_n(x)\equiv s_n\in S'$ we obtain an operator of the form \eqref{tm}.

\section{The regularity property of Hausdorff-type operators}\label{regularity}\setcounter{equation}{0}

The next theorem  gives some  generic  scheme for   generalizations and analogs of  classical results  mentioned above. As regards filters we refer to \cite{GenTopI}. 

We need the following definitions.

\begin{definition}\label{df:regular}
 Let $\frak{F}$ and  $\frak{F}'$ be  filters  on sets $S$ and $S'$ respectively, $A(u):S\to S'$  for all $u\in \Omega$, and $V$ some Banach space.
 We say that  a transformation $\mathcal{H}_{\Phi,A,\mu}$  is
regular with respect to filters  $\frak{F}$ and $\frak{F}'$, and a Banach space $V$ if  for every bounded   function $f:S'\to V$ such that the  mapping  $u\mapsto f(A(u)(x))$ is $\mu$-measurable for each $x \in S$ the equality  $\lim_{x',\frak{F}'} f(x') = l$  in $V$  implies $\lim_{x,\frak{F}}(\mathcal{H}_{\Phi,A,\mu} f)(x) = l$ in $V$.
\end{definition}

\begin{definition}\label{df:filter}
 Let $\frak{F}$ and  $\frak{F}'$ be  filters  on sets $S$ and $S'$ respectively. We say that  a family  of mappings $A(u):S\to S'$ ($u\in \Omega$) 
\textit{agrees with  filters } $\frak{F}$ and $\frak{F}'$ if  $A(u)(\frak{F})$ is a base of  $\frak{F}'$ for $\mu$-a.e. $u\in \Omega$.

If $S=S'$, $\frak{F}=\frak{F}'$, and   $A(u)(\frak{F})$ is a base of  $\frak{F}$ for $\mu$-a.e. $u\in \Omega$ we say that  this family agrees with  $\frak{F}$. 
\end{definition}

\begin{theorem}\label{th:2}
Suppose that the conditions of Definition \ref{general1} are fulfilled,   a filter  $\frak{F}$ on $S$ has a countable base, and  a family of mappings $A(u):S\to S'$ ($u\in \Omega$) agrees with  filters  $\frak{F}$ and $\frak{F}'$. Let the kernel $\Phi$ satisfies the following conditions:

(i) 
$$
a) \sup_{x\in S}\int_{\Omega} |\Phi(u,x)|d\mu(u)<\infty; \  \  b) \forall u\in \Omega \ \sup_{x\in S}|\Phi(u,x)|<\infty;
$$

(ii)  for every $\varepsilon>0$, there are such $K_\varepsilon\subseteq \Omega$ with $\mu(K_\varepsilon)<\infty$ and $F_\varepsilon\in  \frak{F}$  that %for all $x\in S$
$$
\sup_{x\in F_\varepsilon}\int_{\Omega\setminus K_\varepsilon} |\Phi(u,x)|d\mu(u)<\varepsilon;
$$

(iii) for every $\varepsilon>0$, there is such $\delta>0$ that for all $E\subset \Omega$ with $\mu(E)<\delta$
$$
\sup_{x\in S}\int_{E} |\Phi(u,x)|d\mu(u)<\varepsilon.
$$

Then the transformation $\mathcal{H}_{\Phi,A,\mu}$  is
regular  with respect to  $\frak{F}$   and $\frak{F}'$  and every Banach space $V$ if and only if  the condition

(iv)
\begin{equation*}%\label{norm}
\lim_{x,\frak{F}}\int_{\Omega} \Phi(u,x)d\mu(u)=1
\end{equation*}
 holds.
\end{theorem}

\begin{proof} Let the conditions (i) --- (iv) are fulfilled. First note that by Remark \ref{rem:1} integral in \eqref{general} exists in the sense of Bochner for every $x\in S$.

Now let  $\lim_{x',\frak{F}'} f(x') = l$ in norm of some Banach space $V$. Then  for each $u\in \Omega$,
\begin{equation}\label{lim}
\lim_{x,\frak{F}} f(A(u)(x)) = l
\end{equation}
in norm of $V$, as well.  Indeed,   for every  $\varepsilon>0$, there exists such $M'_\varepsilon \in\frak{F}'$ that $\|f(y)-l\|<\varepsilon$ for all $y\in M'_\varepsilon$. Let  $N_\varepsilon \in\frak{F}$ be such that  $M'_\varepsilon\supseteq A(u)(N_\varepsilon)$. Then  $\|f(A(u)(x))-l\|<\varepsilon$ for all $x\in N_\varepsilon$  and \eqref{lim} follows. 
 
 Further, since
 \begin{align*}
 (\mathcal{H}_{\Phi,A,\mu}f)(x)-l &=\int_{\Omega} \Phi(u,x)(f(A(u)(x))-l)d\mu(u)\\
 &+l\left(\int_{\Omega} \Phi(u,x)d\mu(u)-1\right),
 \end{align*}
we have
 \begin{align}\label{est:1}
\|(\mathcal{H}_{\Phi,A,\mu}f)(x)-l\| &\le\int_{\Omega} |\Phi(u,x)|\|f(A(u)(x))-l\|d\mu(u)\\ \nonumber
 &+\|l\|\left|\int_{\Omega} \Phi(u,x)d\mu(u)-1\right|\\\nonumber
 &=I_1(x)+I_2(x).\nonumber
 \end{align}

Let $\varepsilon>0$. In view of (iv) there is such $M_\varepsilon \in\frak{F}$ that $I_2(x)<\varepsilon$ for all $x\in M_\varepsilon$.

Next, for  $K\subseteq \Omega$ with $\mu(K)<\infty$, one has
 \begin{align}\label{est:2}
I_1(x) &=\int_{\Omega\setminus K} +\int_{K}|\Phi(u,x)|\|f(A(u)(x))-l\|d\mu(u)\\\nonumber
&=I_3(x)+I_4(x).\nonumber
 \end{align}
%Let \|f(y)\|\le C for all $y\in S$. 
If $\|f(y)\|\le C$ for all $y\in S'$ then 
 \begin{align}\label{est:3}
|\Phi(u,x)|\|f(A(u)(x))-l\| \le (C+\|l\|)|\Phi(u,x)|.
 \end{align}

By (ii), one can choose  such $K=K_\varepsilon $ of finite $\mu$-measure and $F_\varepsilon\in \frak{F}$ that  $I_3(x)<\varepsilon$ for all $x\in F_\varepsilon$.

Now we claim that 
 \begin{align}\label{I_4to0}
\lim_{x,\frak{F}} I_4(x) =0
 \end{align}
 by the Lebesgue-Vitali Theorem (e.g., \cite[Theorem 4.5.4]{Bogachev}). For the proof of \eqref{I_4to0} note  that the estimate \eqref{est:3} and
 the condition (iii) imply that the family of functions 
  \begin{align}\label{family}
  (|\Phi(\cdot,x)|\|f(A(\cdot)(x))-l\|)_{x\in S}
  \end{align}
  has uniformly absolutely continuous integrals in the sense of  \cite[Definition 4.5.2]{Bogachev}. Moreover, the condition (i)  implies that this family is bounded in $L^1(\mu)$. Then, by \cite[Proposition  4.5.3]{Bogachev}, the family \eqref{family} is uniformly integrable and \eqref{I_4to0} follows  in view of \eqref{lim} by the Lebesgue-Vitali Theorem  (one can apply this theorem, since the  base of $\frak{F}$ is  countable). 
 Thus,  for fixed $K=K_\varepsilon$, there is $B_\varepsilon \in\frak{F}$ that $I_4(x)<\varepsilon$ for all $x\in B_\varepsilon$, and the regularity is proved.

Conversely, if $\mathcal{H}_{\Phi,A,\mu}$ is regular then putting $f(x)\equiv l$ we get (iv).
\end{proof}

\begin{corollary}\label{cor:discrete1} 
Suppose that    a filter  $\frak{F}$ on $S$ has a countable base, and  a sequence  of mappings $A_n:S\to S'$ ($n\in \mathbb{Z}_+$) agrees with $\frak{F}$ and $\frak{F}'$. Let  sequences   $c_n(x)$ ($x\in S$) and $\mu_n>0$ satisfy the following conditions:

($i_d)$
  the series 
$$
 \sum_{n=0}^\infty|c_n(x)|\mu_n
$$
converges  on $S$ to a bounded function;

($ii_d$)  for every $\varepsilon>0$, there are such $K_\varepsilon\subset  \mathbb{Z}_+$, 
with $\sum_{n\in K_\varepsilon}\mu_n<\infty$, and $F_ \varepsilon\in \frak{F}$ that
$$
\sup_{x\in F_ \varepsilon}\sum_{n\in  \mathbb{Z}_+\setminus K_\varepsilon}|c_n(x)|\mu_n<\varepsilon;
$$

($iii_d$) for every $\varepsilon>0$, there is such $\delta>0$ that for all $E\subset \mathbb{Z}_+$ with $\sum_{n\in E}\mu_n<\delta$ one has
$$
\sup_{x\in S}\sum_{n\in E}|c_n(x)|\mu_n<\varepsilon.
$$

Then the transformation $\mathcal{H}_{c,A,\mu}$ given by  \eqref{discrete} is
regular  with respect to  $\frak{F}$ and $\frak{F}'$  and every Banach space $V$  if and only if the condition 

($iv_d$)
\begin{equation*}%\label{norm}
\lim_{x,\frak{F}}\sum_{n=0}^\infty c_n(x)\mu_n=1.
\end{equation*}
  holds.
\end{corollary}

\begin{proof} If
$$
C:= \sup_{x\in S}\sum_{n=0}^\infty|c_n(x)|\mu_n,
$$
then $\sup_{x\in S}|c_n(x)|<C/\mu_n$, and the condition (i) of Theorem \ref{th:2} where $\Omega=\mathbb{Z}_+$, $\mu(\{n\})=\mu_n$, and  $\Phi(n,x)=c_n(x)$ holds. The validity of other conditions of this theorem with $\Omega =\mathbb{Z}_+$, $c_n(x):=\Phi(n,x)$, $A_n(x):=A(n,x)$, and $\mu_n:=\mu(\{n\})$ ($n\in \mathbb{Z}_+$)  is obvious. 
\end{proof}

\begin{corollary}\label{cor:discrete2} 
Suppose that    a filter  $\frak{F}$ on $S$ has a countable base, and  a sequence  of mappings $A_n:S\to S'$ ($n\in \mathbb{Z}_+$) agrees with $\frak{F}$ and $\frak{F}'$. Let  sequences   $c_n(x)$ ($x\in S$) and $\mu_n>0$ satisfy the following condition:

($v_d)$
 $\mu_n\downarrow 0$, and  the series 
$$
 \sum_{n=0}^\infty|c_n(x)|
$$
converges  on $S$ to a bounded function.

Then the transformation $\mathcal{H}_{c,A,\mu}$ given by  \eqref{discrete} is
regular  with respect to  $\frak{F}$ and $\frak{F}'$ and every Banach space $V$ if and only if the condition 
($iv_d$)   holds.
\end{corollary}

\begin{proof} The condition ($v_d)$ implies by the Dirichlet test for function series that conditions  ($i_d)$ and ($ii_d$) are valid. The condition ($iii_d)$ is valid as well, since if
$$
\sup_{x\in S} \sum_{n=0}^\infty|c_n(x)|=:C,
$$
then we have
$$
\sup_{x\in S}\sum_{n\in E}|c_n(x)|\mu_n\le C\sum_{n\in E}\mu_n.
$$

\end{proof}

Now we are aimed to consider the following slightly more general class of operators.

\begin{definition}\label{df:T} Let the conditions of Definition \ref{general1} are fulfilled, $S=S'$, and  $a:S\to \mathbb{C}$ be a function. By a Hausdorff-type operator of a second kind we mean the following transformation
\begin{equation}\label{eq:Ta}  
T_{a,\Phi,A,\mu}f= T_af=af+\mathcal{H}_{\Phi,A,\mu}f.
\end{equation}
\end{definition}

\begin{corollary}\label{cor:T} Suppose that the conditions of Definition \ref{general1} are fulfilled,  $S=S'$, a filter  $\frak{F}$ on $S$ has a countable base, and  a family of mappings $A(u):S\to S$ ($u\in \Omega$) agrees with $\frak{F}$. Let the kernel $\Phi$ satisfies the conditions (i)---(iii) of Theorem \ref{th:2}, the function $a$ is bounded, and the limit $\alpha:=\lim_{x,\frak{F}}a(x)$ exists. In order that the  Hausdorff-type operator of the second kind \eqref{eq:Ta}
 should be 
regular  with respect to $\frak{F}$  and every Banach space $V$  (i.e. that $\lim_{x,\frak{F}}f(x)= l$ in norm of $V$ where $f$ satisfies the conditions of Definition \ref{df:regular} should imply $\lim_{x,\frak{F}}(T_af)(x)= l$ in norm of  $V$), it is necessary and sufficient that

(iv')
\begin{equation*}%\label{norm}
\lim_{x,\frak{F}}\int_{\Omega} \Phi(u,x)d\mu(u)=1-\alpha.
\end{equation*}
\end{corollary}

\begin{proof} Note that
\begin{align*}
(T_af)(x)-l&=a(x)(f(x)-l)+ (\mathcal{H}_{\Phi,A,\mu}f)(x)-(1-a(x))l\\
 &=a(x)(f(x)-l)+\int_{\Omega} \Phi(u,x)(f(A(u)(x))-l)d\mu(u)\\
 &+l\left(\int_{\Omega} \Phi(u,x)d\mu(u)-(1-a(x))\right).
 \end{align*}
If  (i)---(iii)  hold and $\lim_{x,\frak{F}}f(x)=l$ then as was shown in the the proof of Theorem \ref{th:2}
$$
\lim_{x,\frak{F}}\int_{\Omega} \Phi(u,x)(f(A(u)(x))-l)d\mu(u)=0.
$$
In view of (iv') the sufficiency  follows.

Assuming $f(x)\equiv l$, we obtain the necessity of the condition (iv').

\end{proof}

In the following remarks we discuss the conditions  of Theorem \ref{th:2} and Corollary \ref{cor:T}. 

\begin{remark}\label{on(agree)} The necessity of the condition $\mu(\{0\})=0$ in Theorem \ref{th:1} shows that the condition in Theorem \ref{th:2} that the
 family  $(A(u))_{u\in \Omega}$ agrees with  filters  $\frak{F}$ and $\frak{F}'$ cannot be omitted. Indeed, in   Theorem \ref{th:1}  $\Omega =[0,1]$, $S=S'=(0,\infty)$, $\frak{F}=(x\to +\infty)$, $A(u)(x)=ux$, $V=\mathbb{C}$. Thus, the map $A(u)$ agrees with $\frak{F}$ if and only if $u\ne 0$. So, if the
 family  $(A(u))_{u\in \Omega}$ does not agrees with   $\frak{F}$ then $\mu(\{0\})\ne 0$ and $H_\mu$ is not regular.
\end{remark}

\begin{remark}\label{on(i)}
For the discrete measure $\mu$ the condition (i) in  Theorem \ref{th:2} may be necessary. Indeed,
 in the case $\Omega=S=S'=\mathbb{Z}_+$, $\mu(\{n\})\equiv 1$, $\frak{F}=\frak{F}'=(n\to\infty)$, and   $A_n(x)\equiv n\in \mathbb{Z}_+$, this condition is necessary for the regularity of the transformation \eqref{discrete} by the Toeplitz-Shur-Silverman  Theorem \cite[Chapter III, \S 3.2, Theorem 2]{H}.

 Surprisingly,  if the measure $\mu$ is atomless the condition  a) in (i)  can be omitted because in this case  the uniform integrability of the family \eqref{family} (which guarantees  the application of the Lebesgue-Vitali Theorem) is equivalent to the uniform absolute continuity of integrals   (see \cite[Proposition  4.5.3]{Bogachev}) which follows from the condition (iii).

\end{remark}

\begin{remark}\label{on(ii)} If the measure $\mu$ is finite the condition (ii) is  plainly satisfied.
\end{remark}

\begin{remark}\label{on(iii)}
If the kernel $\Phi$ is a bounded function  the condition (iii) is  plainly satisfied, too.
\end{remark}

\begin{remark} The conditions (i) --- (iii) in  Theorem \ref{th:2} follow from the next condition

(v) there is such  $\varphi\in L^1(\mu)$  that
$$
|\Phi(u,x)|\le \varphi(u)\mbox{ for all   } x\in S \mbox{  and } u\in \Omega.
$$

\end{remark}

Let us consider some  examples of  applications of  Theorem \ref{th:2}.

\begin{example}\label{ex:hardy} Consider the operator \eqref{hardy}. In this case  $S=S'=(0,\infty)$, $\frak{F}=(x\to +\infty)$,  $\Omega=[0,1]$,  $A(u)(x)=ux$, $\mu$ is finite. Since $\Phi\equiv 1$,  the conditions (i)---(iii) of Theorem \ref{th:2} hold.   The family of mappings  $(A(u))$ agrees with $\frak{F}$ if and only if  $\mu(\{0\})=0$. The condition (iv) of Theorem \ref{th:2} holds if and only if  $\mu([0,1])=1$.
\end{example}

The next example deals with Hausdorff operators on topological groups. A special case of this result appeared in  \cite{JOTH}.

Let  $S=S'=G$ be a    locally compact  group,  $A(u)\in {\rm Aut}(G)$, and the  group $\mathrm{Aut}(G)$  of all topological automorphisms of  $G$ is equipped   with its natural (Braconnier) topology. In this topology the sets 
$$
\mathcal{O}(C, W):=\{A\in \mathrm{Aut}(G): A(x)x^{-1}\in W,   A^{-1}(x)x^{-1}\in W \forall x\in C\} 
$$
where     $C$ runs over  all
compact subsets of $G$ and $W$ runs over all neighborhoods of the unit in $G$
constitute a fundamental system of neighborhoods of the identity  (see, e.~g., \cite[(26.1)]{HiR}, \cite[Section III.3]{Hoch}).

We are going to apply Theorem \ref{th:2} to the following special case. Let  $f:G\to V$. We say that a vector $l\in V$ is a limit of $f$ as $x\to \infty$ (and write  $\lim_{x\to \infty}f(x)=l$) if $\|f(x)-l\|$ vanishes outside compact subsets of $G$, in other words,  if $\lim_{x,\frak{F}_\infty} f(x) = l$ where  $\frak{F}=\frak{F}'=\frak{F}_\infty$ is the filter on $G$ whose base $\frak{B}_\infty$ consists of all nonempty  complements of compact subsets of $G$.

Recall that a topological space $X$ is said to be $\sigma$-compact (compact at the infinity  in  terminology of N. Bourbaki) if $X$ is a union of a sequence of compact sets.

\begin{theorem}\label{group}  Let $\Omega$ be a topological space with  Borel measure $\mu$, and  $G$ be a locally compact  $\sigma$-compact group. Assume that a family   of topological  automorphisms $A:\Omega\to {\rm Aut}(G)$ is  continuous.  Then under the conditions (i) ---  (iv)  the transformation $\mathcal{H}_{\Phi,A,\mu}$ is regular  with respect to $\frak{F}_\infty$  and every Banach space $V$.

Conversely, if $\mathcal{H}_{\Phi,A,\mu}$  is regular  with respect to $\frak{F}_\infty$  and some nontrivial Banach space $V$, the equality (iv) holds. 
\end{theorem}

\begin{proof} We shall show that the filter $\frak{F}_\infty$ in this case has a countable base.  Since  $G$  is  compact at the infinity, there is an increasing  sequence $(U_n)$ of  open subsets of $G$ with compact closure $ \overline{U}_n$  such that  $G=\cup_{n=1}^\infty U_n$ (e.g., \cite[Chapter I, \S 9, Proposition 15]{GenTopI}). It is known that  every compact subset of $G$ is contained in some $U_n$ (e.g.,  \cite[Chapter I, \S 9, Corollary 1 of Proposition 15]{GenTopI}). It follows that the countable set $\{G\setminus  \overline{U}_n: n\in\mathbb{N}\}$ is a base of $\frak{F}_\infty$.

Further, since each set of the form  $A(u)(G\setminus K)=G\setminus A(u)(K)$ where $K$ is a compact subset of $G$ belongs to  $\frak{B}_\infty$, one has   that $A(u)(\frak{F}_\infty)$ is a base of  $\frak{F}_\infty$ for every  $u\in \Omega$. Thus, the family  $(A(u))_{u\in \Omega}$ agrees with $\frak{F}_\infty$.

Finally, since    for each $x\in G$ the map $\phi\mapsto \phi(x)$, ${\rm Aut}(G)\to G$ is continuous with respect to the
Braconnier topology \cite[Proposition III.3.1, p. 40]{Hoch},  the  map  $u\mapsto A(u)(x)$, $\Omega\to G$ is continuous (and thus  Borel measurable), as well. So, the map  $u\mapsto f(A(u)(x))$ is $\mu$-measurable  for each $x\in G$, since $f$ is  Borel measurable,  and  all conditions of Theorem  \ref{th:2} are fulfilled.
\end{proof}

\begin{example}\label{delsarte}  Let $G$ be a locally compact topological group and $\Omega$
a compact subgroup of $\mathrm{Aut}(G)$ with normalized Haar measure $\mu$. The generalized   shift operator of Delsarte is
$$
(T_hf)(x)=\int_{\Omega} f(hu(x))d\mu(u)\quad (x,h\in G)
$$
(\cite{Delsarte}, see also \cite[Chapter I, \S 2]{Lev}). This is an operator of the form $\mathcal{H}_{1}S_h$ where $S_hf(x)=f(hx)$  is a usual left shift of a function $f:G\to \mathbb{C}$ and
$$
(\mathcal{H}_{1}f)(x)=\int_{\Omega} f(u(x))d\mu(u)
$$
is a Hausdorff-type operator over $G$ where $\Phi(u,x)\equiv 1$ and $A(u)= u$.  The operator $\mathcal{H}_{1}$ satisfies all the conditions of Theorem \ref{group}. Therefore  for  a bounded Borel measurable  function $f$ one has  $\lim_{x\to \infty}(T_h f)(x)=l$  for all $h\in G$ whenever $\lim_{x\to \infty}f(x)=l$.

\end{example}

Since sometimes the language of nets (or sequences) is more convenient that of filters one, we shall give a   
version of  Theorem \ref{th:2} in terms of nets. 

Recall that  a net $(x_i)_{i\in I}$ in  a topological space $S$  approaches the infinity ($x_i\to\infty$ in symbols) if for every compact $K\subset S$ there is such $i_K\in I$ that $x_i\in S\setminus K$ for all $i\ge i_K$.

In the following we write $A(\infty)=\infty$ for a map $A:S\to S$ if for each net  $(x_i)_{i\in I}$   in $S$   such that $x_i\to\infty$ and the partially ordered set $I$  has a countable cofinal part one has $A(x_i)\to\infty$.

\begin{theorem}\label{th:3}
Suppose that the conditions of Definition \ref{general1} are fulfilled and  $S$ is a topological space. Assume that (i)--- (iii) hold and 
 a family $(A(u))_{u\in \Omega}$ of mappings $S\to S$ satisfies $A(u)(\infty)=\infty$ for all $u\in \Omega$.

Let
\begin{equation}\label{norm20}
\lim_{i\in I}\int_{\Omega} \Phi(u,x_i)d\mu(u)=1
\end{equation}
if  $x_i\to\infty$ and $I$  has  a countable cofinal part.

Then the transformation $\mathcal{H}_{\Phi,A,\mu}$  is
regular in the following sense. For every Banach   space $V$ and for every bounded   function $f:S\to V$ such that the  mapping  $u\mapsto f(A(u)(x))$ is $\mu$-measurable for each $x\in S$ the equality  $\lim_{i\in I} f(x_i) = l$  where  $x_i\to\infty$ and $I$  has a countable cofinal part  implies $\lim_{i\in I}(\mathcal{H}_{\Phi,A,\mu} f)(x_i) = l$.

Conversely, if $\mathcal{H}_{\Phi,A,\mu}$  is regular  for some nontrivial Banach space $V$, the equality \eqref{norm20}  holds.
\end{theorem}

The proof of this theorem is similar to the proof of Theorem \ref{th:2} (one can apply the Lebesgue-Vitali Theorem in this case, too, since the  partially ordered set $I$  has a countable cofinal part).

 Let $(A_u)_{u\in \Omega}$ be a Borel measurable family of non-singular real matrices of order $n$,  $\mathbb{R}^n_{>0}:=(0,+\infty)^n$, and $b:\Omega\to \mathbb{R}^n_{>0}$ be some Borel measurable  map. In the next  corollary we consider a Hausdorff-type operator of the form
$$
(H_{\Phi,A,\mu} f)(x)=\int_\Omega \Phi(u,x)f(A_u x+b(u))d\mu(u)
$$
($x\in \mathbb{R}^n$ is a column vector). We write $A_u>0$ if $A_u$ is a matrix with positive elements. We write $x\to+\infty$ if $x\in \mathbb{R}^n_{>0}$ and $x\to\infty$.

\begin{corollary}(cf. \cite{faa})  Let $\Omega$ be a topological space with a $\sigma$-finite  Borel measure $\mu$, $A_u\in \mathrm{GL}(n, \mathbb{R})$ and both  maps $b:\Omega\to \mathbb{R}^n_{>0}$, and  $u\mapsto A_u:\Omega\to \mathrm{GL}(n, \mathbb{R})$ are Borel measurable. Assume that each $A_u>0$ and conditions (i) --- (iii), and \eqref{norm20} hold. Then $H_{\Phi,A,\mu}$ is
regular in the following sense.  If  $f$ is a bounded Borel measurable $V$-valued function on  $\mathbb{R}^n_{>0}$ then
$\lim_{x\to+\infty}(H_{\Phi,A,\mu} f)(x) = l$ if    $\lim_{x\to+\infty} f(x) = l$. 

Conversely, if $H_{\Phi,A,\mu}$  is regular for some nontrivial Banach space $V$, the equality  \eqref{norm20}  holds.
\end{corollary}

\begin{proof} In our case $S=\mathbb{R}^n_{>0}$, $A(u)(x)=A_ux+b(u)$, and one can use  sequences instead of nets.  Note that each $A(u)$ maps $\mathbb{R}^n_{>0}$ into itself. Since  $|A_u x|\ge \frac{1}{\|A_u^{-1}\|}|x|$ for all  $x\in \mathbb{R}^n$ (here $\|A_u^{-1}\|$ denotes the operator norm of a matrix, $|x|$ denotes the Euclidean norm in $\mathbb{R}^n$), we have that $x_k\to\infty$ in $S$ implies $A(u)(x_k)\to\infty$ in $S$ ($k\to\infty$). 

 Finally, the map  $u\mapsto f(A(u)(x)+b(u))$ is Borel measurable  for each $x\in  \mathbb{R}^n_{>0}$, since  the map $u\mapsto A_u x$ between  $\Omega$ and $ \mathbb{R}^n_{>0}$ is Borel measurable and $b$ is Borel measurable, too.
\end{proof}

\begin{comment}

\section*{Funding} 
This work was supported in part by
 the State Program of Scientific Research
of Republic of Belarus, project No. 20211776
 and by the Ministry of Education and Science of Russia,  agreement No. 075-02-2025-1720.

\section*{Affiliations}
Department of Mathematics and Programming Technologies, Francisk
Skorina Gomel State University, Gomel, Belarus  $\&$
Regional Mathematical Center, Southern Federal
University, Rostov-on-Don, 344090, Russia

\section*{Ethics declarations}
The author of this work declare that he has no conflicts of interest.

%\section{Data availability statement}
%The author confirms that all data generated or analyzed during this study
%are included in this article.
%This work does not have any conflicts of interest.

\end{comment}

\end{document}